\documentclass{amsart}

\usepackage{amsmath}
\usepackage{amsthm}
\usepackage{amssymb}
\usepackage{mathtools}
\usepackage{mathdots}
\usepackage{amsfonts}
\usepackage{dsfont}
\usepackage{xcolor}
\usepackage[colorlinks=true,linkcolor=black,urlcolor=black,debug]{hyperref}
\usepackage[capitalize]{cleveref}
\usepackage{array}
\usepackage[alphabetic]{amsrefs}
\usepackage[all,cmtip]{xy}
\usepackage{soul}

\usepackage{comment}

\theoremstyle{plain}

 \newtheorem{theorem}{Theorem}[section]
 \newtheorem*{theorem*}{Theorem}
 
 \newtheorem{lemma}[theorem]{Lemma}
 \newtheorem{corollary}[theorem]{Corollary}

\theoremstyle{definition}

\theoremstyle{remark}

 \newtheorem{remark}[theorem]{Remark}
 \newtheorem{example}[theorem]{Example}
 
\def\N{\mathbb{N}}
\def\P{\mathbb{P}}
\def\Z{\mathbb{Z}}

\def\cF{\mathcal{F}}
\def\cG{\mathcal{G}}
\def\cL{\mathcal{L}}
\def\cO{\mathcal{O}}
\def\cQ{\mathcal{Q}}

\def\D{\Delta}
\def\g{\gamma}
\def\r{\rho}
\def\Om{\Omega}

\def\to{\rightarrow}

\DeclareMathOperator{\Mod}{Mod}
\DeclareMathOperator{\Alg}{Alg}
\DeclareMathOperator{\Set}{Set}
\DeclareMathOperator{\HS}{HS}
\DeclareMathOperator{\UHS}{\mathds{HS}}
\DeclareMathOperator{\Hom}{Hom}
\DeclareMathOperator{\Sym}{Sym}
\DeclareMathOperator{\gr}{gr}

\DeclareMathOperator{\id}{id}

\DeclareMathOperator{\Der}{Der}

\DeclareMathOperator{\Coh}{Coh}

\DeclareMathOperator{\Spec}{Spec}

\newcommand{\dual}[1]{#1^\vee} 

\begin{document}

  \title{Higher derivations of modules and the Hasse-Schmidt module}
  
  \author{Christopher Chiu}
  
  \address[C.\ Chiu]{%
    Fakult\"at f\"ur Mathematik\\
    Universit\"at Wien\\
    Oskar-Morgenstern-Platz 1\\
    A-1090 Wien (\"Osterreich)%
  }
  \email{christopher.heng.chiu@univie.ac.at}

  \author{Luis Narv\'aez Macarro}
  
  \address[L. Narv\'aez Macarro]{
  Departamento de \'Algebra \& Instituto de Matem\'aticas (IMUS)\\
    Facultad de Matem\'aticas, 
    Universidad de Sevilla\\
    Calle Tarfia s/n\\
    E-41012 Sevilla (Espa\~na)
  }
  \email{narvaez@us.es}
  

  \thanks{
   The research of the first author was supported by project P-31338 of the Austrian Science Fund (FWF). The research of the second author was partially supported by MTM2016-75027-P and FEDER.}
  
  \begin{abstract}
   In this paper we revisit Ribenboim's notion of higher derivations of modules and relate it to the recent work of De Fernex and Docampo on the sheaf of differentials of the arc space. In particular, we derive their formula for the K\"ahler differentials of the Hasse-Schmidt algebra as a consequence of the fact that the Hasse-Schmidt algebra functors commute.
  \end{abstract}

  \maketitle

Higher derivations of modules were introduced in \cite{rib_higher_der} in analogy to higher derivations of rings to provide a similar notion of a map carrying ``infinitesimal information''. In particular, 
there too exists a universal object parametrizing such higher derivations of modules, which we call the \emph{Hasse-Schmidt module}. This construction was implicitly considered in \cite{def_doc_differentials} when establishing a formula for the sheaf of differentials of jet and arc spaces. More precisely, the main statement behind the aforementioned formula in the affine case is the following:

\begin{theorem*}[\cite{def_doc_differentials}, Theorem 5.3]
 Let $A$ be a $k$-algebra. For $n\in\N\cup\{\infty\}$ denote by $\UHS_k^n(A)$ the $n$-th Hasse-Schmidt algebra of $A$. Then there exists an $A$-module $Q_n$ such that
 \[
  \Om_{\UHS_k^n(A)/k}=\Om_{A/k}\otimes_A Q_n.
 \]
\end{theorem*}

As the $n$-th Hasse-Schmidt algebra $\UHS_k^n(A)$ parametrizes infinitesimal data on $A$ up to order $n$ (i.e. \emph{$n$-jets} on $A$), this formula suggests that tangents (i.e. infinitesimal data up to order $1$) of $n$-jets on $A$ can be recovered from some higher order operation on tangents on $A$. We aim to make this idea precise by observing that $\Om_{A/k}\otimes_A Q_n$ is just the $n$-th Hasse-Schmidt module of $\Om_{A/k}$. This was essentially proven in \cite{def_doc_differentials} but missing in \cite{rib_higher_der}; we will derive it here by considering $Q_n$ as a dual bimodule, respectively, as a colimit of duals in the case $n=\infty$. The formula itself can thus be seen as a consequence of the (elementary) fact that the Hasse-Schmidt algebra functors commute, see \cref{l:hs-alg-bigrad}. The final ingredient is the following theorem, which connects the Hasse-Schmidt algebra to the Hasse-Schmidt module by means of the symmetric algebra.

\begin{theorem*}
 Let $M$ be an $A$-module and denote by $\UHS^n_{A/k}(M)$ the $n$-th Hasse-Schmidt module of $M$. Then
 \[
  \Sym_{A_n}(\UHS^n_{A/k}(M))\simeq \UHS^n_k(\Sym_A(M)).
 \]
\end{theorem*}

Similar to how Hasse-Schmidt algebras of finite order can be glued to give the jet space $X_n$ of an algebraic variety $X$, gluing Hasse-Schmidt modules gives a sheaf of modules over $X_n$. In particular, this construction yields a whole class of ``universal'' vector bundles on $X_n$.

The basic outline of this paper is as follows: in \cref{s:hs-algebras} we will collect well-known facts about higher derivations and extend them to graded higher derivations. In \cref{s:dual-bimodules} we will briefly touch on the notion of dualizable bimodules. In \cref{s:hs-modules} we discuss higher derivations of modules and the Hasse-Schmidt module as well as its relationship to the Hasse-Schmidt algebra, including proving the aforementioned result as \cref{t:hs-sym-alg}. Finally, in \cref{s:univ-vec-bundles} we will briefly talk about globalizing the construction in \cref{s:hs-modules} to obtain sheaves on jet and arc spaces.

Throughout this paper, $k$ will denote an arbitrary commutative ring. If $R$ is any other ring, we will write $R[[t]]_n:=R[[t]]/(t^{n+1})$ for $n\in \N$ and $R[[t]]_\infty=R[[t]]$. The category of $k$-algebras will be denoted by $\Alg_k$; furthermore, we write $\Alg_{k,\gr}$ for the category of $\N$-graded $k$-algebras (where we regard $k$ with the trivial grading). A morphism in $\Alg_{k,\gr}$ will always be a graded map of degree $0$.

\section{Higher derivations and the Hasse-Schmidt algebra}

\label{s:hs-algebras}

In this section we will briefly recall some well-established facts about higher derivations (also called Hasse-Schmidt derivations) and the universal object associated to them, which we call the Hasse-Schmidt algebra and denote by $\UHS_k^n(A)$. Our main references here are \cite{vojta} and \cite{rib_hs_I}, \cite{rib_hs_II}. Furthermore, we will extend these definitions to graded rings and show that the corresponding universal object can be obtained by endowing $\UHS_k^n(A)$ with an induced (natural) grading.

\subsection{Higher derivations}

If $A$, $C$ are $k$-algebras and $n\in \N\cup\{\infty\}$, then a higher derivation $D=(D_i)_{i=0}^n:A\to C$ of order $n$ is a collection of $k$-linear maps $D_i:A\to C$ such that $D_0$ is a map of $k$-algebras and the \emph{higher Leibniz rules} are satisfied:
\[
 D_i(ab)=\sum_{k+l=i} D_k(a)D_l(b).
\]
We write $D\in \HS^n_k(A,C)$. There exist bijections
\begin{equation}
\label{e:nat-bij-hs}
 \HS^n_k(A,C)\simeq \Hom_{\Alg_k}(A,C[[t]]_n),
\end{equation}
which are natural in $A$ and $C$. The image of $D$ under this map will be denoted by $\g_D$.

If $C$ has in addition an $A$-algebra structure, then we write $\HS_{A/k}^n(A,C)$ for the subset of all higher derivations $D$ such that $D_0:A \to C$ is the structure map $a\mapsto a\cdot 1_C$. The natural bijections \cref{e:nat-bij-hs} restrict to
\begin{equation}
 \label{e:nat-bij-hs-res}
  \HS_{A/k}^n(A,C) \simeq \Hom_{\Alg_k}^\circ(A,C[[t]]_n),
\end{equation}
where the right-hand side denotes the subset of $k$-algebra maps $\g: A\to C[[t]]_n$ such that $\g$ modulo $(t)$ equals the structure map $A\to C$.

For any $k$-algebra $A$ the functor $\HS^n_k(A,-)$ is representable by a $k$-algebra $\UHS^n_k(A)$, the Hasse-Schmidt algebra, which comes equipped with a universal higher derivation $d_A = (d_{A,i}): A\to \UHS_k^n(A)$. The map $A\to \UHS_k^n(A)[[t]]_n$ corresponding to $d_A$ under \cref{e:nat-bij-hs} will be denoted by $\g_A$. By definition for every $k$-algebra $C$ there exists natural bijections
\begin{equation}
\label{e:nat-bij-hs-alg}
 \Hom_{\Alg_k}(\UHS_k^n(A),C) \simeq \HS_k^n(A,C) \simeq \Hom_{\Alg_k}(A,C[[t]]_n),
\end{equation}
given by
$$
\varphi \in \Hom_{\Alg_k}(\UHS_k^n(A),C) \longmapsto \widetilde{\varphi} \circ \gamma_A \in \Hom_{\Alg_k}(A,C[[t]]_n),
$$
where $\widetilde{\varphi}: \UHS_k^n(A)[[t]]_n \to C[[t]]_n$ is the $t$-linear extension of $\varphi $.
We write $\varphi_D$ for the map $\UHS_k^n(A)\to C$ corresponding to a higher derivation $D\in \HS_k^n(A,C)$: $D_i = \varphi_D \circ d_{A,i}$ for all $i$.
For convenience's sake we will often write $A_n:=\UHS^n_k(A)$.

Similarly, for every $A$-algebra $C$ we obtain bijections
\begin{equation}
\label{e:nat-bij-hs-alg-res}
 \Hom_{\Alg_A}(\UHS_k^n(A),C) \simeq \HS_{A/k}^n(A,C) \simeq \Hom_{\Alg_k}^\circ(A,C[[t]]_n),
\end{equation}
which are natural in $C$.

\begin{remark}
 \label{r:hs-alg-funct}
 The assignment $A \mapsto \UHS^n_k(A)$ yields a functor $\Alg_k \to \Alg_k$; in particular, any $k$-algebra map $f: A\to A'$ gives rise to a natural map $f_n: \UHS^n_k(A) \to \UHS^n_k(A')$ such that $f_n \circ d_{A,i} = d_{A',i}$ for all $i$.
\end{remark}

\begin{remark}
 \label{r:truncation-maps}	
 For $m>n$ there exist natural maps $\pi_{m,n}^*: \UHS_k^n(A)\to \UHS_k^m(A)$ obtained from the truncations
 \[
  C[[t]]_m\to C[[t]]_n.
 \]
 We will refer to the maps $\pi_{m,n}^*$ as \emph{co-truncation maps}.
\end{remark}

\begin{remark}
 \label{r:hs-alg-direct-limit}	
  The co-truncation maps $\pi_{m,n}^*$ give rise to a directed system and we have $\UHS_k^\infty(A)\simeq \varinjlim_n \UHS_k^n(A)$. Indeed, this follows directly from the universal property of the Hasse-Schmidt algebra, as we have natural bijections
  \[
    \Hom_{\Alg_k}(A,C[[t]])\simeq \Hom_{\Alg_k}(A,\varprojlim_n C[[t]]_n)\simeq \varprojlim_n \Hom_{\Alg_k}(A,C[[t]]_n).
  \]
\end{remark}

The Hasse-Schmidt algebra $\UHS^n_k(A)$ can be constructed as a quotient of the polynomial algebra
\[
 A[x^{(i)} \mid x\in A,\:i=1,\ldots,n]
\]
by the ideal generated by
\begin{align}
\label{e:hs-alg-presentation}
 (x+y)^{(i)} - x^{(i)} - y^{(i)}, \:x,y\in A \\
 (xy)^{(i)}-\sum_{k+l=i} x^{(k)}y^{(l)}, \:x,y\in A \notag \\
 a^{(i)}, \:a\in k, \notag
\end{align}
for $i=1,\ldots,m$; note that we identify $x^{(0)}$ with $x\in A$. See \cite{vojta} for more details. In this presentation the universal higher derivation $d^n_A$ is given by
\[
 d_A(x)=\sum_{i=0}^n x^{(i)} t^i.
\]

\begin{remark}
 \label{r:nat-gr-hs-alg}
 There exists a natural $\N$-grading of $\UHS^n_k(A)$ given by $\deg(x^{(i)})=i$ for $x\in A$, which we will refer to as the \emph{structural grading} of $\UHS^n_k(A)$. Indeed, notice that the system of equations (\ref{e:hs-alg-presentation}) is homogeneous with respect to $\deg(x^{(i)})=i$. Moreover, for any $k$-algebra map $f: A\to A'$ the natural map $f_n: \UHS^n_k(A) \to \UHS^n_k(A')$ from \cref{r:hs-alg-funct} is graded.
\end{remark}

\subsection{Graded higher derivations}

Let us first consider the case $n\in\N$, i.e. of higher derivations of finite order. Let $k$ be a ring and $A=\bigoplus_{i\in\N} A_i$ and $C=\bigoplus_{i\in\N} C_i$ graded $k$-algebras (where we regarded $k$ with the trivial grading; in particular the sets of homogeneous elements $A_i$ and $C_i$ are $k$-modules). We call a higher derivation $D=(D_i)_i:A\to C$ of order $n$ \emph{graded} if every component $D_i$ is graded (of degree $0$), that is, we have $D_i(A_j)\subset C_j$ for all $j\in\N$. The set of all such $D$ will be denoted by $\HS_{k,\gr}^n(A,C)$. Note that $C[[t]]_n$ is a free $C$-module of rank $n+1$ and thus carries a natural grading induced by $C$. It is then immediate that the natural isomorphism \cref{e:nat-bij-hs} restricts to
\[
 \HS_{k,\gr}^n(A,C)\simeq \Hom_{\Alg_{k,\gr}}(A,C[[t]]_n).
\]
We claim that there exists an $\N$-grading on $\UHS_k^n(A)$, different from the structural grading introduced in \cref{r:nat-gr-hs-alg}, such that, for every $\N$-graded $k$-algebra $C$, the natural bijections \cref{e:nat-bij-hs-alg} restrict to natural bijections
\[
\Hom_{\Alg_{k,\gr}}(\UHS_k^n(A),C)\simeq \HS_{k,\gr}^n(A,C) \simeq \Hom_{\Alg_{k,\gr}}(A,C[[t]]_n).
\]
In particular $\g_A:A\to \UHS_k^n(A)[[t]]_n$ will be map of graded $k$-algebras. We call this grading the \emph{induced grading} of $\UHS^n_k(A)$. To construct the induced grading, note that, if $A=\bigoplus_{i\in\N} A_i$, then the presentation given by \cref{e:hs-alg-presentation} can be refined such that $A$ is given as the quotient of
\[
 A[x^{(i)} \mid x\in A_j, \:j\in\N, \:i=1,\ldots,n]
\]
by the ideal generated by
\begin{align}
\label{e:hs-alg-presentation-gr}
 (x+y)^{(i)} - x^{(i)} - y^{(i)}, \:x,y\in A_j \\
 (xy)^{(i)}-\sum_{k+l=i} x^{(k)}y^{(l)}, \:x,y\in A_j \notag \\
 a^{(i)}, \:a\in k, \notag
\end{align}
for $i=1,\ldots,m$ and $j\in\N$. We define the induced grading by setting $\deg(x^{(i)}):=j$ for $x\in A_j$. Note that the system \cref{e:hs-alg-presentation-gr} is homogeneous with respect to this grading, so it is well-defined. The $k$-module $(\UHS_k^n(A))_i$ of elements of degree $i$ is generated by the set of products
\[
 \{x_1^{(j_1)}\cdots x_r^{(j_r)} \mid x_l\in A_{i_l},\: i_1+\ldots+i_r=i\}.
\]
Note that $(\UHS_k^n(A))_0=\UHS^n_k(A_0)$ and $(\UHS^n_k(A))_i$ is a module over $\UHS_k^n(A_0)$.

Now to see the claim above, if $\varphi_D: \UHS_k^n(A)\to C$ is graded, then $\varphi(x^{(j)})\in C_i$ for $x\in A_i$. The map $\g_D$ corresponding to $\varphi_D: A\to C[[t]]_n$ under \cref{e:nat-bij-hs-alg} is given by
\[
 x\mapsto \sum_{j=0}^n x^{(j)} t^j,
\]
and thus is graded. The other direction follows in analogy.

\begin{remark} 
\label{rem:0-part-grHS}
 Under the above hypotheses, for any $D\in \HS_{k,\gr}^n(A,C)$, we consider $D^0 = (D_i^0)_i$, with 
 $D_i^0: A_0 \to C_0$ the degree $0$ part of $D_i:A\to C$. It is clear that $D^0\in \HS_k^n(A_0,C_0)$.
\end{remark}

In fact, taking the Hasse-Schmidt algebra gives rise to a functor $\Alg_{k,\gr} \to \Alg_{k,\gr}$:

\begin{theorem}
\label{t:funct-gr-hs-alg}	
 Let $n\in \N$ and $A=\bigoplus_i A_i$ be a graded $k$-algebra (where we consider $k$ with the trivial grading). Then the assignment $A\mapsto \UHS_k^n(A)$ yields a functor $\Alg_{k,\gr}\to \Alg_{k,\gr}$ satisfying the following: for every $\N$-graded $k$-algebra $C$ there exist bijections
 \begin{equation}
 \label{e:nat-bij-gr-hs-alg}
  \Hom_{\Alg_{k,\gr}}(\UHS_k^n(A),C)\simeq\Hom_{\Alg_{k,\gr}}(A,C[[t]]_n),
 \end{equation}
 natural in $A$ and $C$.
\end{theorem}

\begin{remark}
 \label{r:funct-gr-hs-alg-res}
 Similarly, if $C$ is a graded $A$-algebra and one defines $\HS_{A/k,\gr}^n(A,C)$ to be the subset of graded higher derivations $D: A\to C$ such that $D_0:A\to C$ is the structure map, then there exist natural bijections
 \begin{equation}
  \Hom_{\Alg_{A,\gr}}(\UHS_k^n(A),C)\simeq \HS_{A/k,\gr}^n(A,C) \simeq \Hom^\circ_{\Alg_{k,\gr}}(A,C[[t]]_n).
 \end{equation}
 Note that $\UHS^n_k(A)$ is a graded $A$-algebra by construction.
\end{remark}

Let us now consider the case $n=\infty$. Clearly the co-truncation maps $\pi_{m,n}^*: \UHS_k^n(A)\to \UHS_k^m(A)$, $m>n$ preserve the gradings on both sides. By \cref{r:hs-alg-direct-limit} we obtain a grading on $\UHS_k^\infty(A)$ coming from its direct limit structure. However, note that there exists no result which is directly analogous to \cref{t:funct-gr-hs-alg}, since, for $C$ a graded $k$-algebra, the power series ring $C[[t]]$ does no longer carry a grading compatible with that of $C$ itself. In fact, one would need to interpret $\Hom_{\Alg_{k,\gr}}(A,C[[t]])$ as the set of all $k$-algebra maps $\varphi: A \to C[[t]]$ such that $\varphi(A_i) \subset C_i[[t]]$ for all $i\geq 0$.

\begin{remark}
\label{r:hs-alg-bigr}
 The structural grading of $\UHS^n_k(A)$ (see \cref{r:nat-gr-hs-alg}) together with its induced grading yields an $\N$-bigrading of $\UHS^n_k(A)$. Indeed, observe that in the presentation given by \cref{e:hs-alg-presentation-gr} this bigrading is defined via $\deg(x^{(j)})=(i,j)$ for $x\in A_i$. Moreover, for $f:A\to A'$ a map between graded $k$-algebras, the induced map $f_n: \UHS^n_k(A)\to \UHS^n_k(A')$ clearly respects the bigrading. Hence we can view $A\mapsto \UHS^n_k(A)$ as a functor from $\N$-graded $k$-algebras to $\N$-bigraded $k$-algebras.
\end{remark}

The following lemma states that the functors $\UHS^n_k(-)$ commute with each other. Moreover, the composition of two such functors can also be seen as a universal object for $2$-variate higher derivations; note that for $p$-variate higher derivations the corresponding Hasse-Schmidt algebra 
is $\N^p$-graded. For more details on the $p$-variate case and the proof of the lemma see \cite[Corollary 2.3.12]{nar_course}.

\begin{lemma}
 \label{l:hs-alg-bigrad}
 For all $n,m\in \N\cup\{\infty\}$ there are natural isomorphisms of bigraded $\UHS^n_k(A)$-algebras
 \begin{equation} \label{eq:HS-n-m}
  \UHS^n_k(\UHS^m_k(A))\simeq \UHS^{(n,m)}_k(A)\simeq \UHS^m_k(\UHS^n_k(A))^*,
 \end{equation}
 where, for any bigraded algebra $B$, we denote by $B^*$ the bigraded algebra obtained from interchanging the order of gradings on $B$. Moreover, if we call:
 \begin{itemize}
  \item $d^m_{A,i}:A \to \UHS^m_k(A)$, $i\leq m$, the components of the universal higher derivation $d^m_A \in \HS^m_k(A,\UHS^m_k(A))$, and
  \item $d^{(n,m)}_{A,(i,j)}: A \to  \UHS^{(n,m)}_k(A)$, $(i,j)\leq (n,m)$, the components of the universal higher derivation $d^{(n,m)}_A \in \HS^{(n,m)}_k(A,\UHS^{(n,m)}_k(A))$,
 \end{itemize}
then the isomorphisms in (\ref{eq:HS-n-m}) are determined by 
\[
 d^n_{\UHS^m_k(A),i} \circ d^m_{A,j} \equiv d^{(n,m)}_{A,(i,j)} \equiv d^m_{\UHS^n_k(A),j} \circ d^n_{A,i}.
\] 
\end{lemma}

\section{Dualizable bimodules and limits}

\label{s:dual-bimodules}

Let $A$ be a ring and $M$ a (left) $A$-module, then the (right) $A$-module $\dual{M}:=\Hom_A(M,A)$ is commonly called the \emph{dual module} of $M$. In general, the dual of a module is not well-behaved, as for example $M$ might not be \emph{reflexive}, i.e. the natural map $M\to \dual{(\dual{M})}$ might not be an isomorphism. For a module $M$ to be \emph{dualizable} we will thus require it to have a dual object in the (stronger) categorical sense of \cite{dold_puppe}. In this section we will recall some well-known and elementary facts about dualizable (bi)modules, which will be used to establish the existence of the Hasse-Schmidt module in \cref{s:hs-modules}.

Let us start by fixing some notation. If $M$ and $N$ are both right $A$-modules, we denote the set of right $A$-homomorphisms $M \to N$ by $\Hom_{(-,A)}(M,N)$. Similarly, if $M$ and $N$ are left $B$-modules, we write $\Hom_{(B,-)}(M,N)$ for the set of left $B$-homomorphisms $M \to N$. This choice of notation will allow us to be precise about the type of structure considered when dealing with bimodules.

Let $A$, $B$ be (not necessarily) commutative rings and $P$ a $(B,A)$-bimodule. We say that $P$ is \emph{left dualizable} if there exists an $(A,B)$-bimodule $Q$ and bimodule homomorphisms
\[
 \eta: A \to Q \otimes_B P,\:\theta: P\otimes_A Q \to B
\]
such that the diagrams
\[
 \xymatrix@C+2pc{A \otimes_A Q \ar[r]^{\eta \otimes \id_Q} \ar[d]^{\simeq} & (Q\otimes_B P) \otimes_A Q \ar[d]^{\simeq}\\
           Q \otimes_B B & \ar[l]_{\id_Q \otimes \theta} Q \otimes_B (P \otimes_A Q)}
\]
and
\[
 \xymatrix@C+2pc{P \otimes_A A \ar[r]^{\id_P \otimes \eta} \ar[d]^{\simeq} & P \otimes_A (Q \otimes_B P) \ar[d]^{\simeq}\\
           B \otimes_B P & \ar[l]_{\theta \otimes \id_P} (P \otimes_A Q) \otimes_B P,}
\]
commute, with the vertical arrows being the associators and unitors. The map $\theta$ is called the \emph{evaluation} and the map $\eta$ the \emph{coevaluation}. Furthermore, the $(A,B)$-bimodule $Q$ is unique up to isomorphism and will be referred to as the \emph{left dual} of $P$. If there is no ambiguity we will write $\dual{P}=Q$ and $\dual{Q}=P$. Note that $P$ being left dualizable is equivalent to saying that the functor $-\otimes_B P$ from right $B$-modules to right $A$-modules is right adjoint to $-\otimes_A Q$; in particular, by tensor-hom adjunction we have a natural isomorphism
\[
 P\simeq \Hom_{(-,B)}(Q,B).
\]

We will now recall this classical result, which characterizes dualizable bimodules as those who are finite projective.

\begin{theorem}
\label{t:dual-bimod}
 Let $P$ be a $(B,A)$-bimodule. Then $P$ is left dualizable if and only if it is finitely generated projective as a left $B$-module. In particular, we have bijections
 \[
  \Hom_{(-,A)}(M,N\otimes_B P) \simeq \Hom_{(-,B)}(M \otimes_A \dual{P},N),
 \]
 which are natural in $M$ and $N$.
\end{theorem}

\begin{remark}
 If a functor $G$ is left adjoint to $-\otimes_B P$, then $G$ is automatically of the form $-\otimes_A Q$ for some $(A,B)$-bimodule $Q$ (see for example \cite[Theorem 3.60]{schommer_pries}). Thus $P$ being left dualizable is equivalent to $-\otimes_B P$ being right adjoint.
\end{remark}

\begin{remark}
 If $\mathcal{C}$ is a monoidal category (for example, the category of modules over a commutative ring $R$), then a dualizable object $X$ of $\mathcal{C}$ is an adjoint to the morphism corresponding to $X$ in the delooping $2$-category $\mathbf{B}\mathcal{C}$. Similarly, a bimodule $P$ is left dualizable if the corresponding $1$-morphism in the bicategory of algebras, bimodules and intertwiners (as introduced in \cite{benabou}) is a right adjoint.
\end{remark}

For the purposes of the next section we need to consider the case where $P$ is a cofiltered limit of $(B,A)$-bimodules $P_i$, which are finitely generated projective as left $B$-modules. In general, $P$ will not be finite projective over $B$ itself, so there does not exist a (left) dual of $P$ in the sense discussed above. However, the functor $\varprojlim_i (-\otimes_B P_i)$ does have a left adjoint, which just follows from the fact that right adjoints commute with limits. We summarize this fact in the following corollary:

\begin{corollary}
\label{c:dual-hom-formula-inf}
 Let $P=\varprojlim_i P_i$ be a cofiltered limit of $(B,A)$-bimodules $P_i$ which are finitely generated projective over $B$. Then the left duals $\dual{P_i}$ form a filtered system. Furthermore, we have bijections
 \[
  \Hom_{(-,A)}(M,\varprojlim_i N\otimes_B P_i) \simeq \Hom_{(-,B)}(M \otimes_A \varinjlim_i \dual{P_i},N),
 \]
 which are natural in $M$ and $N$.
\end{corollary}

\section{Higher derivations of modules and the Hasse-Schmidt module}

\label{s:hs-modules}

Higher derivations of modules were first introduced in \cite{rib_higher_der}, where they were defined with respect to a given higher derivation of rings. The existence of a universal object, which we call Hasse-Schmidt module, parametrizing such higher derivations was already established there; however, it also appeared implicitly in more detail in \cite{def_doc_differentials}. In this section our aim is to provide a top-down view of the construction of the Hasse-Schmidt module and how it relates to the Hasse-Schmidt algebra by means of \cref{t:hs-sym-alg}, which is the main result of this section.

Throughout this section we will write $A_n:=\UHS_k^n(A)$ for the Hasse-Schmidt algebra of a $k$-algebra $A$. All rings considered here are assumed to be commutative; in particular, we do not have to distinguish between left and right actions.

Let now $A$ and $C$ be $k$-algebras and $D=(D_i)_{i=0}^n:A\to C$ be a higher derivation of length $n$ over $k$. A \emph{higher derivation over $D$},  $\D=(\D_i)_{i=0}^n: M\to N$, where $M\in\Mod_A$, $N\in\Mod_C$, is a collection of $k$-linear maps $\D_i:M\to N$ satisfying
\[
 \D_i(a\cdot m)=\sum_{k+l=i} D_k(a)\cdot \D_l(m),\:a\in A, m\in M.
\]
The set of all such maps is denoted by $\HS^n_D(M,N)$. We have isomorphisms
\[
 \HS^n_D(M,N)\simeq \Hom_A(M,N[[t]]_n),
\]
which are natural in $M$ and $N$. Note that the $A$-module structure on $N[[t]]_n$ above comes from scalar restriction through $\g_D:A\to C[[t]]_n$, with $D_i = \varphi_D \circ d^n_{A,i}$. If we consider $N$ as an $A_n$-module via restriction by $\varphi_D: A_n \to C$, then we see that
\[
 \HS^n_D(M,N)\simeq \HS^n_{d^n_A}(M,N).
\]
Let us first treat the case $n\in \N$. We consider $A_n[[t]]_n$ as a $(A_n,A)$-bimodule, where the left action is just the usual multiplication and the right action is induced by $\g_A$. Then we have an isomorphism of right $A$-modules
\[
 N[[t]]_n \simeq N \otimes_{A_n} A_n[[t]]_n.
\]
Note that $A_n[[t]]_n$ is free of rank $n+1$ over $A_n$. By \cref{t:dual-bimod} we get
\[
 \HS^n_{d^n_A}(M,N)\simeq \Hom_A(M,N\otimes_{A_n} A_n[[t]]_n)\simeq \Hom_{A_n}(M\otimes_A \dual{(A_n[[t]]_n)},N),
\]
where $\dual{(A_n[[t]]_n)}\simeq\Hom_{(-,A_n)}(A_n[[t]]_n,A_n)$ is the left dual of $A_n[[t]]_n$. Using the fact that extension of scalars is left adjoint to restriction, we obtain
\[
 \HS^n_D(M,N)\simeq \Hom_C(M\otimes_A \dual{(A_n[[t]]_n)}\otimes_{A_n} C,N)
\]

Now, in the case $n=\infty$, we have that $A_\infty[[t]]=\varprojlim_n A_\infty[[t]]_n$ is the limit of a projective system of $(A_\infty,A)$-bimodules which are finite free over $A_\infty$. Arguing as before and applying \cref{c:dual-hom-formula-inf} yields
\[
 \HS^\infty_D(M,N)\simeq \Hom_{C}(M\otimes_A \varinjlim_n \dual{(A_\infty[[t]]_n)} \otimes_{A_\infty} C,N).
\]
Thus we have proven the following result:

\begin{theorem}[{\cite[\S 4]{rib_higher_der}}]
 \label{t:universal-hs-module}
 Let $A$ and $C$ be $k$-algebras and $D:A\to C$ be a higher derivation of order $n\in \N \cup \{\infty\}$. Then, for any $A$-module $M$, the functor
 \[
  \HS^n_D(M,-): \Mod_{C} \to \Set
 \]
 is representable by a module $\UHS_{A/k}^n(M)\otimes_{A_n} C$, where we call the $A_n$-module \\ $\UHS_{A/k}^n(M)$ the \emph{$n$-th Hasse-Schmidt module} of $M$. Moreover, we have
 \[
  \UHS_{A/k}^n(M)\simeq M\otimes_A \dual{(A_n[[t]]_n)}
 \]
 for $n\in \N$ and
 \[
  \UHS_{A/k}^\infty(M)\simeq M\otimes_A \varinjlim_n \dual{(A_\infty[[t]]_n)},
 \]
 where $\dual{(A_\infty[[t]]_n)}$ is the (left) dual of the finite free $A_\infty$-module $A_\infty[[t]]_n$.
\end{theorem}

\begin{remark}
 \label{r:universal-hs-module-der}
 The module $\UHS_{A/k}^n(M)$ comes attached with a universal higher derivation $\D_M: M\to \UHS_{A/k}^n(M)$ over $d_A^n: A\to A_n$. We want to give an explicit description of $\D_M$ in the case where $M=A$, i.e. $\UHS_{A/k}^n(M)=\dual{(A_n[[t]]_n)}$. To that avail, if $n\in\N$, observe that the $A$-module map $\alpha_D: A\to \UHS_{A/k}^n(M)[[t]]_n$ corresponding to $\D_M$ is just the coevaluation. Thus, if we take the standard basis $t^i$, $i=0,\ldots,n$ for $A_n[[t]]_n$ over $A_n$ and let $t^{[i]}$, $i=0,\ldots,n$, denote the dual basis, the map $\alpha_D$ is given by
 \begin{equation}
  \label{e:universal-hs-module-der-2}
  1_A \mapsto \sum_{i=0}^n t^{[i]} t^i.
 \end{equation}
 Note that if $n=\infty$ then the images of $t^{[i]}$, $i\in\N$, in $\varinjlim_n \dual{(A_\infty[[t]]_n)}$ form a basis over $A_\infty$. Hence the universal higher derivation $\D_A: A\to \UHS_{A/k}^n(M)$ is given by
 \begin{equation}
  \label{e:universal-hs-module-der}
  (\D_M)_i(1_A)=t^{[i]}.
 \end{equation}
\end{remark}

\begin{remark}
 \label{r:dual-isomorphic}
 For $n\in \N$, since $A_n[[t]]_n$ is finite free over $A_n$, we have that
 \[
  \UHS_{A/k}^n(M) \simeq M \otimes_A A_n[[t]]_n
 \]
 as $A_n$-modules. Indeed, choosing the same basis as in \cref{r:universal-hs-module-der}, this isomorphism is given identifying $A_n[[t]]_n$ with its dual via $t^i\mapsto t^{[i]}$.
\end{remark}

\begin{remark}
 \label{r:hs-module-proj-free}
 As $A_n[[t]]_n$ is free for $n\in\N\cup\{\infty\}$, if $M$ is projective respectively free then so is $\UHS_{A/k}^n(M)$.
\end{remark}

\begin{remark}
 \label{r:universal-hs-module-extension}
 Alternatively one could consider triples $(C,D,N)$, with $C$ a $k$-algebra, $D\in\HS_k^n(A,C)$ and $N$ a $C$-modules, a suitable notion of morphism between such triples and the functor which associates to each such triple the set $\HS^n_D(M,N)$. Then the argument given before shows that this functor is represented by the triple $(A_n,d^n_A,\UHS_{A/k}^n(M))$.
\end{remark}

\begin{remark}
 \label{r:hs-mod-direct-limit}	
 We have natural isomorphisms $\UHS^\infty_{A/k}(M)\simeq \varinjlim\limits_n \UHS^n_{A/k}(M)$. Indeed, this follows just as in \cref{r:hs-alg-direct-limit} from
 \[
  \Hom_A(M,N[[t]])\simeq \Hom_A(M,\varprojlim_n N[[t]]_n)\simeq \varprojlim_n \Hom_A(M,N[[t]]_n),
 \]
 where $N$ is an $A_n$-module.
\end{remark}

We will refer to \cite{rib_higher_der} for a basic introduction to the theory of higher derivations of modules. As one fact not included there, let us mention that the Hasse-Schmidt module behaves well under base change.

\begin{lemma}
 \label{l:hs-module-base-change}
 Let $A\to A'$ be a $k$-algebra map and $M$ an $A$-module. Then we have
 \[
  \UHS_{A'/k}^n(M\otimes_A A')\simeq \UHS_{A/k}^n(M)\otimes_{A_n} A'_n.
 \]	
\end{lemma}

\begin{proof}
 Follows immediately from the description in \cref{t:universal-hs-module}.
\end{proof}	

The following theorem is the main result of this section and relates the Hasse-Schmidt module to the Hasse-Schmidt algebra by means of the symmetric algebra.

\begin{theorem}
 \label{t:hs-sym-alg}	
 For each $n\in\N\cup\{\infty\}$ there are isomorphisms of $\N$-graded $A_n$-algebras
 \[
  \Sym_{A_n}(\UHS^n_{A/k}(M))\simeq \UHS^n_k(\Sym_A(M)),
 \]
 which are natural in $M$. Note that on the left side we are considering the natural grading on the symmetric algebra, while on the right hand side we are taking the induced grading of $\UHS^n_k(\Sym_A(M))$.
\end{theorem}

\begin{remark}
 By considering $\UHS^n_{A/k}(M)$ as a graded $A_n$-module similar to \cref{r:nat-gr-hs-alg} the isomorphism in \cref{t:hs-sym-alg} extends to one of $\N$-bigraded algebras. 
\end{remark}

\begin{proof}
  Consider first the case $n\in \N$ and let $B=\bigoplus_{i\geq0} B_i$ be an $\N$-graded $A_n$-algebra, where we consider $A_n$ with the trivial grading. Then each $B_i$ is in particular an $A_n$-module. Recall that $B[[t]]_n$ is a free $B$-module of rank $n+1$ and thus has a natural $\N$-grading given by $(B[[t]]_n)_i:=B_i[[t]]_n$. We obtain natural bijections
  \begin{align*}
   \Hom_{\Alg_{A_n}^{\gr}}(\Sym_{A_n}(\UHS^n_{A/k}(M)),B) & \simeq \Hom_{\Mod_{A_n}}(\UHS^n_{A/k}(M),B_1) \simeq \\
   \simeq \Hom_{\Mod_A}(M,B_1[[t]]_n) & \simeq \Hom_{\Alg_A^{\gr}}(\Sym_A(M),B[[t]]_n).
  \end{align*}
  Here again we consider $A$ with the trivial grading. If $\rho: A_n \to B$ is the map defining the $A_n$-algebra structure on $B$, then the $A$-algebra (resp. $A$-module) structure on $B[[t]]_n$ (resp. $B_1[[t]]_n$) is given by $\tilde{\rho}\circ d_A^n$, where $\tilde{\rho}: A_n[[t]]\to B[[t]]$ is obtained from $\rho$ by $t$-linear extension.
  We claim that 
  \[
   \Hom_{\Alg_A^{\gr}}(\Sym_A(M),B[[t]]_n) \simeq \Hom_{\Alg_{A_n}^{\gr}}(\UHS^n_k(\Sym_A(M)),B).
  \]
  Indeed, an element $\alpha$ of the left-hand side is given by a triangle 
  \[
   \xymatrix{\Sym_A(M) \ar[r]^-\alpha & B[[t]]_n \\
             A \ar[u] \ar[ur]_{\tilde{\rho}\circ d_A^n} & }
  \]
  where $\alpha=(\alpha_i)_{i\in\N}$ is graded of degree $0$. By \cref{t:funct-gr-hs-alg} we obtain a triangle of $k$-algebra maps
  \[
   \xymatrix{\UHS^n_k(\Sym_A(M)) \ar[r]^-{\alpha^*} & B \\
             \UHS^n_k(A) \ar[u] \ar[ur]_{\rho} & },
  \]
  where $\alpha^*$ and $\rho$ is graded. Conversely, every such triangle is obtained by one of the form above. This proves the claim for $n\in \N$.

  We are thus left with the case $n=\infty$. Taking colimits of the isomorphism for finite $n$ we obtain
  \[
  \varinjlim_n \UHS^n_k(\Sym_A(M)) \simeq \varinjlim_n \Sym_{A_n}(\UHS^n_{A/k}(M)) \simeq \Sym_{A_\infty}(\varinjlim_n \UHS^n_{A/k}(M)).
  \]
  As we have both $\UHS^\infty_k(B)=\varinjlim_n \UHS^n_k(B)$ and $\UHS^\infty_{A/k}(M)=\varinjlim_n \UHS^n_{A/k}(M)$ the result follows.
\end{proof}

\begin{remark}
 \label{r:dualizing-module}
 Applying \cref{t:hs-sym-alg} in the case $M=A$ and $n\in \N$ yields the following:
 \begin{equation}
 \label{e:desc-universal-hs-module}
  \UHS_{A/k}^n(A)\simeq \dual{(A_n[[t]]_n)} \simeq [\UHS_k^n(\Sym_A(A))]_1.
 \end{equation}
 To make this isomorphism explicit, let $e:=1_A$. Then 
 \[
  \UHS_k^n(\Sym_A(A))\simeq A_n[e^{(i)} \mid i=0,\ldots,n].
 \]
 Thus the $A_n$-submodule of elements of degree $1$ is generated by $e^{(i)}$, $i=0,\ldots,n$. Using the same basis as in \cref{r:universal-hs-module-der} we see that the isomorphism in \cref{e:desc-universal-hs-module} is given by
 \[
  t^{[i]}\mapsto e^{(i)}.
 \]
 Now, since the $A$-action on $\UHS_{A/k}^n(A)\simeq \dual{(A_n[[t]]_n)}$ is given by precomposition, we obtain an induced $A$-action on $[\UHS_k^n(\Sym_A(A))]_1$ given by
 \[
  a\cdot e^{(i)}=\sum_{j=0}^i a^{(i-j)} e^{(j)}.
 \]
 Compare this with the construction of $P_n$ given in \cite[Section 4]{def_doc_differentials}. It is not clear how to obtain this $A$-action naturally by just considering $\UHS_k^n(\Sym_A(A))$.
\end{remark}

Using \cref{t:hs-sym-alg} we will recover the formula in \cite[Theorem 5.3]{def_doc_differentials} as a direct consequence of the fact that the Hasse-Schmidt algebra functors commute (see \cref{l:hs-alg-bigrad}).

\begin{theorem}
 \label{c:hs-cotangent}	
 For all $n\in\N \cup\{\infty\}$ there is a natural isomorphism
 \[
  \Om_{A_n/k}\simeq \UHS^n_{A/k}(\Om_{A/k})\simeq \Om_{A/k}\otimes_A \UHS_{A/k}^n(A).
 \]	
 of $A_n$-modules. Additionally, if $n\in \N$, we have
 \[
  \Om_{A_n/k}\simeq \Om_{A/k}\otimes_A A_n[[t]]_n.
 \]
\end{theorem}

\begin{proof}
 Note that $\Sym_B(\Om_{B/k})\simeq \UHS^1_k(B)$ for any $k$-algebra $B$. Thus \cref{l:hs-alg-bigrad} implies that we have isomorphisms of graded algebras
 \[
  \UHS^n_k(\Sym_A(\Om_{A/k}))\simeq \UHS^n_k(\UHS^1_k(A))\simeq \UHS^1_k(\UHS^n_k(A))\simeq \Sym_{A_n}(\Om_{A_n/k}).
 \]
 Thus the statement follows from \cref{t:hs-sym-alg} for $M=\Om_{A/k}$ and taking the homogeneous part of degree $1$ on each side. The second assertion just follows from \cref{r:dual-isomorphic}.
\end{proof}

\begin{remark}
\label{r:der-as-higher-der}
 A similar argument allows us to consider (usual) $k$-derivations $d:A\to M$ with $M$ an $A$-module as higher derivations in the above sense. To that end, note that
 \[
  \Der_k(A,M)\simeq \Hom_A(\Om_{A/k},M)\simeq \Hom_{\Alg_A,\gr}(\Sym_A(\Om_{A/k}),\Sym_A(M)).
 \]
 Since $\Sym_A(\Om_{A/k})\simeq \UHS^1_k(A)$ as graded $k$-algebras, by \cref{r:funct-gr-hs-alg-res} we obtain
 \[
  \Der_k(A,M)\simeq \Hom^\circ_{\Alg_{k,\gr}}(A,\Sym_A(M)[[t]]_1).
 \]
 Since every such map factors as $A\to \UHS^1_k(A)[[t]]_1\to \Sym_A(M)[[t]]_1$ we have
 \[
  \Hom^\circ_{\Alg_{k,\gr}}(A,\Sym_A(M)[[t]]_1)\hookrightarrow \Hom_{\Mod_A}(A,\Sym_A(M)[[t]]_1),
 \]
 where the $A$-action on the $\Sym_A(M)[[t]]_1$ is induced by $d^1_A$. Thus we obtain our claim that
 \[
  \Der_k(A,M)\hookrightarrow \HS^1_{d^1_A}(A,\Sym_A(M)).
 \]
\end{remark}

\begin{remark}
\label{r:HS-der-mod-as-HS-der-sym-alg} 
In a similar vein to \cref{t:hs-sym-alg}, 
higher derivations of modules can be seen as graded higher derivations of symmetric algebras. Namely, if $A$ and $C$ are $k$-algebras, $M\in\Mod_A$, $N\in\Mod_C$, $D\in\HS^n_k(A,C)$ and $n\in \N$, we have:
\begin{gather*}
 \HS^n_D(M,N)  \simeq \Hom_A(M,N[[t]]_n) \simeq
 \Hom_{\Alg_A,\gr}(\Sym_A(M),\Sym_{C[[t]]_n}(N[[t]]_n))\\
 \simeq \Hom_{\Alg_A,\gr}(\Sym_A(M),\Sym_C(N)[[t]]_n) \subset \Hom_{\Alg_k,\gr}(\Sym_A(M),\Sym_C(N)[[t]]_n) \\
 \simeq \HS^n_{k,\gr}(\Sym_A(M),\Sym_C(N)), 
\end{gather*}
and so any $\Delta \in \HS^n_D(M,N)$ gives rise to a well-defined graded higher derivation $\mathbf{\Delta} \in \HS^n_{k,\gr}(\Sym_A(M),\Sym_C(N))$. Notice that for the degree $0$ part we have $\mathbf{\Delta}^0=D$. Actually, the above procedure gives natural bijections
\[
\{ (D,\Delta)\ |\ D\in \HS^n_k(A,C), \Delta\in \HS^n_D(M,N)\} \simeq \HS^n_{k,\gr}(\Sym_A(M),\Sym_C(N)).
\]
Since these bijections are compatible with truncations, they also hold for $n=\infty$.
\end{remark}

\section{Universal vector bundles on jet spaces}
\label{s:univ-vec-bundles}

Let $X$ be a scheme of finite type over a field $k$. For $n\in \N$ the \emph{$n$-th jet space} $X_n$ of $X$ is the $k$-scheme representing the functor 
\[
 Y\mapsto \Hom_k(Y\times_k \Spec(k[[t]]_n),X).
\]
The natural maps $k[[t]]_{n+1}\to k[[t]]_n$ give rise to an inverse system $X_{n+1}\to X_n$. The \emph{arc space} $X_\infty$ of $X$ is defined to be the limit $X_\infty=\varprojlim\limits_n X_n$, which is again a scheme. If $X=\Spec(A)$ is affine, then so is $X_n$ for all $n\in \N\cup \{\infty\}$. In fact, we have $X_n=\Spec(A_n)$, with $A_n=\UHS^n_k(A)$ the $n$-th Hasse-Schmidt algebra. For general $X$ the scheme $X_n$ is constructed by considering a covering of $X$ by affine opens $\Spec(A)$ and glueing the corresponding Hasse-Schmidt algebras $A_n$ to obtain a sheaf of $\cO_X$-algebras, whose relative spectrum gives $X_n$.

Let $n\in \N$. Then the $n$-th jet space $X_n$ comes equipped with a universal family $\g: U_n:=X_n \times_k \Spec(k[[t]]_n)\to X$. Write $\r$ for the canonical projection $U_n\to X_n$. If $X=\Spec(A)$ is affine, then $\g$ is the map induced by the homomorphism $\g_A:A\to \UHS^n_k(A)[[t]]_n$ and $\r$ is induced by the inclusion $\UHS^n_k(A)\to \UHS^n_k(A)[[t]]_n$. 

Let $\cF$ be a coherent sheaf of $\cO_X$-modules. For $n\in \N\cup\{\infty\}$ we will construct a coherent sheaf of $\cO_{X_n}$-modules $\cF_n$ by glueing the Hasse-Schmidt modules of $\cF(U)$ for $U\subset X$ affine. As in the construction of $X_n$, to be able to glue we use that $S^{-1}\UHS_{A/k}^n(M)\simeq \UHS^n_{S^{-1}A/k}(S^{-1}M)$ for any multiplicative subset $S\subset A$, which in turn follows from \cref{l:hs-module-base-change}. Since for any $W\subset X$ open affine the scheme $W_n\subset X_n$ is open affine we obtain a coherent sheaf $\cF$ on $X_n$.

Alternatively, as done in \cite{def_doc_differentials}, for $n\in \N$ we may construct from $Q_n=\UHS_{A/k}^n(A)=\dual{(A_n[[t]]_n)}$ a sheaf $\cQ_n$ on $U_n=X_n\times_k \Spec(k[[t]]_n)$. It is then easy to check that $\cF_n$ verifies the following universal property:

\begin{theorem}
 Let $n\in \N$, $X$ a scheme of finite type over $k$ and $X_n$ the $n$-th jet space of $X$. For any $\cF\in\Mod_{\cO_X}$ the functor
 \[
  \cG\in \Mod_{\cO_{X_n}} \mapsto \Hom_{\cO_X}(\cF,\g_{*}\r^*\cG)
 \]
 is represented by $\cF_n=\r_{*}(\g^*\cF\otimes \cQ_n)$.
\end{theorem}

If $\cF\in \Mod_{\cO_X}$ is locally free, then so is $\cF_n \in \Mod_{\cO_{X_n}}$ by \cref{r:hs-module-proj-free}. In this case, the global version of \cref{t:hs-sym-alg} says the following:

\begin{lemma}
 Let $X$ be a scheme of finite type over a field $k$ and $n\in \N$. Then the assignment $\cF\mapsto \cF_n$ gives a functor $\Coh(X)\to \Coh(X_n)$ which preserves the property of being locally free. Moreover, if $E\to X$ is the vector bundle corresponding to a locally free $\cF\in\Coh(X)$, then its jet space $E_n\to X_n$ is the vector bundle corresponding to $\cF_n\in\Coh(X_n)$.
\end{lemma}

\begin{remark}
 A similar result holds true for the arc space $X_\infty$ of $X$, only in this case the universal family attached to $X_\infty$ is of the form $\widehat{U_\infty}\to X$, where $\widehat{U_\infty}$ is the formal completion of the scheme $X_\infty\times_k \Spec(k[[t]])$ with respect to the closed subscheme $X_\infty \times \{0\}$.
\end{remark}

The construction of the sheaf $\cF_n$ was used in \cite{def_doc_differentials} to provide a global variant of \cref{c:hs-cotangent}, as well as in \cite{def_doc_nash} to show how the jet space behaves under Nash blow-up. Our hope is that other, similarly useful, formulas can be obtained by studying universal vector bundles on jet spaces, in particular in constructing compactifications. 

We finish with computing explicitly $\cF_n$ for an invertible sheaf $\cF$ on $\P^1$, showing that even in this very easy case the resulting universal vector bundle has potentially interesting structure.

\begin{example}
 Let $X=\P^1$ and and $\cL$ an invertible sheaf on $\P^1$, isomorphic to $\cO(d)$ for some $d\in \Z$. Let $U_0$ and $U_1$ be the standard affine opens of $\P^1$. We write $U_i=\Spec(k[t_i])$, then $\cL \mid_{U_i}$ is generated by $e_i$ and the transition map $\cL \mid_{U_i\cap U_j} \to \cL \mid_{U_j\cap U_i}$ is given by $e_0\mapsto t_1^d e_1$. For $n\in \N$ we want to describe explicitly the sheaf $\cL_n$ on $(\P^1)_n$. To that avail, note that $(\P^1)_n$ has an affine covering given by $(U_i)_n=\Spec(k[t_i^{(0)},\ldots,t_i^{(n)}])$. The restriction $\cL_n \mid_{(U_i)_n}$ is freely generated by the sections $e_i^{(j)}$ for $j=0,\ldots,n$. By functoriality of the Hasse-Schmidt module the transition map is given by
 \[
  e_0^{(j)} \mapsto \D_j\left( (t_1^{(0)})^d e_1^{(j)} \right), 
 \]
 where $\D$ is the universal higher derivation associated to $\cL_n \mid_{(U_1)_n}$. In particular, for $d=1$, it is easy to see that the space of global sections of $\cO(1)_n$ is generated by $e_i^{(j)}$, where $i=0,1$ and $j=0,\ldots,n$. This gives a notion of ``coordinates'' for $(\P^1)_n$. The same obviously holds for any projective space $\P^m$, $m\geq1$.
\end{example}


\begin{thebibliography}{99}


\bibitem{benabou}
J.~B\'enabou.  
\newblock \href{https://doi.org/10.1007/BFb0074299}{Introduction to bicategories}.
\newblock{\em Reports of the Midwest Category Seminar, Lect. Notes in Math. 47}, Springer (1967), 1--77.

\bibitem{def_doc_differentials}
T.~de Fernex, R.~Docampo.  
\newblock \href{https://doi.org/10.1215/00127094-2019-0043}{Differentials on the arc space}.
\newblock{\em Duke Math. J.}  169, no. 2 (2020), 353--396.
\newblock ({\tt  \href{https://arxiv.org/abs/1703.07505}{arXiv:1703.07505}}).

\bibitem{def_doc_nash}
T.~de Fernex, R.~Docampo.  
\newblock \href{https://doi.org/10.5802/aif.3302}{Nash blow-ups of jet schemes}.
\newblock{\em Ann. Inst. Fourier}  69, no. 6 (2019), 2577--2588.
\newblock ({\tt  \href{https://arxiv.org/abs/1712.00911}{arXiv:1712.00911}}).

\bibitem{dold_puppe}
A.~Dold, D.~Puppe.
\newblock Duality, trace and transfer.
\newblock{\em Proceedings of the International Conference on Geometric Topology (Warsaw, 1978)} (1980), 81--102.

\bibitem{nar_course}
L.~Narv\'{a}ez~Macarro. 
\newblock \href{https://personal.us.es/narvaez/course_HS_2018.pdf}{A mini-course on Hasse-Schmidt derivations}.
\newblock Lecture notes, 2018.

\bibitem{rib_hs_I}
P.~Ribenboim.  
\newblock Higher derivations of rings, I.
\newblock{\em Revue Roumaine de Math. Pures et Appl.} 16 (1971), 77--110.

\bibitem{rib_hs_II}
P.~Ribenboim.  
\newblock Higher derivations of rings, II.
\newblock{\em Revue Roumaine de Math. Pures et Appl.} 16 (1971), 245--272.

\bibitem{rib_higher_der}
P.~Ribenboim.  
\newblock Higher order derivations of modules.
\newblock{\em Portugal. Math.} 39, no. 1-4 (1980), 381--397.

\bibitem{schommer_pries}
C.~Schommer-Pries.
\newblock The Classification of Two-Dimensional Extended Topological Field Theories.
\newblock {\em PhD Dissertation 2009}.
\newblock ({\tt  \href{https://arxiv.org/abs/1112.1000v2}{arXiv:1112.1000}}).

\bibitem{vojta}
P.~Vojta.
\newblock Jets via Hasse-Schmidt Derivations.
\newblock{\em Diophantine Geometry, Proceedings, Edizioni della Normale, Pisa} (2007), 335--361.
\newblock ({\tt  \href{https://arxiv.org/abs/math/0407113}{arXiv:0407113}}).

\end{thebibliography}
\end{document}